\documentclass[12pt,reqno]{amsart}

\usepackage[left=2.5cm,top=2.5cm,right=2.5cm,bottom=2.5cm]{geometry}
\usepackage{latexsym,amssymb}

\newtheorem{theorem}{Theorem}
\newtheorem{lemma}[theorem]{Lemma}
\newtheorem{corollary}[theorem]{Corollary}

\DeclareMathOperator{\diam}{diam}

\DeclareMathOperator{\capt}{capt}

\begin{document}

\title{Topological directions in Cops and Robbers}
\author{Anthony Bonato}
\address{Department of Mathematics\\
Ryerson University\\
Toronto, ON, Canada} \email{abonato@ryerson.ca}
\author{Bojan Mohar}
\address{Department of Mathematics\\
Simon Fraser University\\
Burnaby, BC, Canada} \email{mohar@sfu.ca}

\thanks{The authors gratefully acknowledge support from NSERC. The second author is also supported in part by the Canada Research Chairs program, and by the Research Grant P1-0297 of ARRS (Slovenia).}


\begin{abstract}
We present the first survey of its kind focusing exclusively on results at the intersection of topological graph theory and the game of Cops and Robbers, focusing on results, conjectures, and open problems for the cop
number of a graph embedded on a surface. After a discussion on results for planar graphs, we consider graphs of higher genus. In 2001, Schroeder conjectured that if a graph has genus $g,$ then its
cop number is at most $g + 3.$ While Schroeder's bound is known to hold for planar and toroidal graphs, the case for graphs with higher genus remains open. We consider the capture time of graphs on
surfaces and examine results for embeddings of graphs on non-orientable surfaces. We present a conjecture by the second author, and in addition, we survey results for the lazy cop number, directed
graphs, and Zombies and Survivors.
\end{abstract}

\maketitle

\section{Introduction}

Topological graph theory is a well-developed and active field, and one of its main questions concerns properties of graphs embedded on surfaces.  Graph embeddings correspond to drawings of graphs on
a given surface, where two edges may only intersect in a common vertex. The topic of planar graphs naturally fits within the broader context of graph embeddings on surfaces. See the book \cite{MT} for
further background on the topic. Graph searching is a rapidly emerging area of study within graph theory, and one of the most prominent directions there is in the study of the game of Cops and
Robbers and its variants. The cop number is the main graph parameter in this area, and its consideration has lead to a number of open problems, especially Meyniel's conjecture on the maximum
asymptotic order of the cop number of a connected graph. See the book \cite{bonato} for an introduction to Cops and Robbers and cop number and \cite{bonato1} for background on Meyniel's conjecture.

The present survey considers the intersection of topological graph theory and Cops and Robbers, with the aim of cataloguing results and highlighting major open problems. While the interaction of the
two areas is only now gaining more prominence, its study traces back to the 1984 paper of Aigner and Fromme \cite{af}, which proved that the cop number of a connected planar graph is at most three.
Since that early work, there have been several works on the cop number of planar graphs, outerplanar graphs, and graphs on surfaces (both orientable and non-orientable ones).

\subsection{Topological considerations}
The reader is directed to \cite{MT} for further background. We define a \emph{surface} as a connected compact 2-dimensional manifold with boundary. In particular, \emph{closed surfaces} (surfaces whose boundary is empty) are connected Hausdorff topological spaces in which every point has a neighborhood that is homeomorphic to the plane $\mathbb{R}^2$.  Closed surfaces are classified by their
genus and orientability type. A surface is \emph{orientable} if it is possible to assign a local sense of clockwise and anticlockwise so that along any path between any two points in the surface the
local sense is consistent; otherwise, it is \emph{non-orientable}. Equivalently, a surface is \emph{non-orientable} if it contains a subset that is homeomorphic to the M\"obius band. The
\emph{surface classification theorem} states that any closed surface is homeomorphic to exactly one of the following surfaces: a sphere, a finite connected sum of tori, or a sphere with a finite
number of disjoint discs removed and with cross-caps (that is, M\"obius bands) glued in their place. The number of tori in the orientable case or cross-caps in the non-orientable case equals the
\emph{genus} of the surface. For example, the sphere has genus 0 and the torus genus 1, the projective plane has genus $1$, while the Klein bottle has genus $2.$

The (non-orientable) genus of a graph $G$ is the minimum $g$ such that $G$ is embeddable in a surface of (non-orientable) genus $g$. We write $\gamma(G)$ for the genus of $G$, and
$\widetilde{\gamma}(G)$ for the non-orientable genus of $G$. If $g=0$, then we say the graph is \emph{planar}. Planar graphs possess a rich corpus within graph theory; for example, the famous
$4$-Color Theorem states that the chromatic number of any planar graph is at most 4. Graphs with genus 1 are referred to as \emph{toroidal}.

We note the connection with surfaces and graph minors. A graph $H$ formed from $G$ by first taking a subgraph (that is, removing some vertices and edges from $G$) and then contracting some of the
remaining edges is said to be a \emph{minor} of $G$. Kuratowksi's theorem says that a graph is planar if and only if it has neither $K_5$ nor $K_{3,3}$ as a minor. Generalizing this result, Robertson
and Seymour proved in \cite{GM8} that for each surface $S$, there is a finite set ${\mathcal F}(S)$ of graphs such that any graph $G$ embeds in $S$ if and only if it has no minor in ${\mathcal
F}(S)$.

\subsection{Cops and Robbers} We next turn to Cops and Robbers, which is a game of perfect information (that is, each player is aware of all the moves of the other player) played on graphs.
For a given graph $G$, there are two players, with one player controlling a set of \emph{cops}, and the second controlling a single \emph{robber}. The game is played over a sequence of discrete
time-steps; a \emph{round} of the game is a move by the cops together with the subsequent move by the robber. The cops and robber occupy vertices of $G$, and when a player is ready to move in a round
the robber or each of the cops may move to a neighboring vertex or stay where they are. The cops move first, followed by the robber; thereafter, the players move on alternate steps. A cop or the
robber can \emph{pass}; that is, stay at their current vertex. Note that the players can occupy any vertex of the graph in their initial move. Observe that any subset of cops may move in a given round. The cops win if after a finite number of rounds, one of them can occupy the
same vertex as the robber. This is called a \emph{capture}. The robber wins if they can evade capture indefinitely. Note that the initial placement of the cops will not affect the outcome of the
game, as the cops can expend finitely many moves to occupy a particular initial placement (the initial placement of the cops may, however, affect the length of the game).

Note that if a cop is placed at each vertex, then the cops are guaranteed to win. Therefore, the minimum number of cops required to win in a graph $G$ is a well-defined positive integer, named the
\emph{cop number} of the graph $G.$  The notation $c(G)$ is used for the cop number of a graph $G$. If $c(G)=k,$ then $G$ is $k$-\emph{cop-win}. In the special case $k=1,$ $G$ is said to be a \emph{cop-win} graph.

For familiar examples, the cop number of any tree is 1 and the Petersen graph is $3$-cop-win. In a graph $G$, a set of vertices $S$ is \emph{dominating} if every vertex of $G$ not in $S$ is adjacent
to some vertex in $S$. The \emph{domination number} of a graph $G$ is the minimum cardinality of a dominating set in $G$. Observe that $c(G)$ is upper bounded by the domination number of $G$;
however, this bound is far from tight. For example, trees are cop-win but can have domination number that is linear in their order (as is the case of paths).

\subsection{Organization} The survey is organized as follows. In Section~\ref{secp}, we describe results and questions on the cop number and capture time
for planar and outerplanar graphs. Schroeder's conjecture is presented in Section~\ref{secg}, where we examine the cop number in higher genus. The cop number in the non-orientable case is also
considered. In Section~\ref{seclow}, we discuss lower bounds on the cop number in terms of the genus. Theorem~\ref{mn} is presented for the first time in published form, and we discuss Mohar's
conjecture. In Section~\ref{secm}, we pull together various topics in Cops and Robbers played on surfaces, all with the common theme of bounds on the cop number. In particular, we examine the game of
Lazy Cops and Robbers, Zombies and Survivors, and the cop number of directed graphs. We finish with a summary of the conjectures and open problems presented in the survey.

All graphs we consider are simple, finite, and undirected unless otherwise stated. As the cop number is additive on components, we will assume throughout that \emph{all graphs are connected}. For
background on graph theory, see \cite{west}.

\section{Planar and outerplanar graphs}\label{secp}

Aigner and Fromme~\cite{af} first introduced the cop number, although the game with one cop only was considered earlier in \cite{nw,q}. They proved an influential result on the cop number of
planar graphs.

\begin{theorem}[\cite{af}]\label{theorem:planar}
If $G$ is a planar graph, then $c(G)\leq 3$.
\end{theorem}

The proof of Theorem~\ref{theorem:planar} is elementary but non-trivial. It makes essential use of what is nowadays referred to as the \emph{Isometric Path Lemma}. Given an induced subgraph $S$ of
$G$, we say that a set of cops \emph{guards} $S$ if whenever the robber enters $S$, they will be caught by one of the cops. A path $P$ in a graph $G$ is \emph{isometric} if for all vertices $x,y$ in
$P$, their distance in $P$ is the same as in the whole graph; that is, $d_P(x,y) = d_G(x,y)$.

\begin{lemma}[\cite{af}]\label{ipl}
If $P$ is an isometric path in $G$, then one cop can guard $P.$
\end{lemma}

Lemma~\ref{ipl} follows by considering a retract of the robber onto $P$. It is sufficient to show the cop can capture the robber's image under the retract.

The idea of the proof of Theorem~\ref{theorem:planar} is to increase the \emph{cop territory}; that is, a set $S$ of vertices such that if the robber moved to $S$, then they would be caught. The cops may begin on a single vertex, and that vertex trivially is cop territory. If the
territory can always be increased, the number of vertices the robber can move to without being caught is eventually reduced to the empty set, and so the robber is captured. The cop does so by
ensuring the cop territory is always one of three types. For each type, we can (by using the Isometric Path Lemma) enlarge the cop territory by induction and remain among the three types. For more
details, the reader is directed to consult the proof presented in the book \cite{bonato} or the original one in \cite{af}. It is interesting to note that the proof only uses
implicitly the Jordan curve theorem and no other deeper properties of planar graphs.

From Theorem~\ref{theorem:planar}, each planar graph is one of three types: those with cop numbers $1$, $2$, or $3$. Cop-win graphs are precisely the dismantlable ones, and that characterization may eventually aid
in the classification of planar cop-win graphs. Nevertheless, the classification of which planar graphs have cop number $i$, for $1\le i \le 3$, remains an open problem.

We also do not understand small order planar graphs with cop number 3. It has been proved in \cite{bevbon} that the smallest order $3$-cop-win graph is the Petersen graph. A conjecture from \cite{bevbon}
is that the smallest order $3$-cop-win planar graph is the dodecahedron.

A graph is \emph{outerplanar} if it can be embedded in the plane so that all its vertices are on the outer face. Outerplanar graphs are precisely those which do not contain a $K_{2,3}$ and $K_4$ as a
minor. The following result was first proven by Clarke in her doctoral thesis.

\begin{theorem}[\cite{clarket}]\label{outer}
If $G$ is an outerplanar graph, then $c(G)\leq 2$.
\end{theorem}

The proof of Theorem~\ref{outer} is simpler than the one for Theorem~\ref{theorem:planar}. For a sketch, the proof breaks into the case when $G$ is $2$-connected or not. In the $2$-connected case,
one first observes that the boundary of the outer face is a Hamiltonian cycle, and the two cops successively enlarge their territory by moving along the cycle towards the robber. The case where there
are cut-vertices involves first a retraction of blocks to cut-vertices, and then using the strategy from the $2$-connected case to chase the robber into a block where the robber cannot escape.

As with planar graphs, the classification of outerplanar graphs which are $1$- and $2$-cop-win remains open. Of course, the smallest outerplanar graph with cop number 2 is the four cycle.

The \emph{length} of a Cops and Robbers game played on $G$ is the number of rounds it takes (not including the initial round) to capture the robber on $G$. We say that a play of the game with $c(G)$
cops is \emph{optimal} if its length is the minimum over all possible strategies for the cops, assuming the robber is trying to evade capture for as long as possible. There may be several optimal
plays possible (for example, on a path with an even number of vertices, the cop may start on either of the two vertices in the center), but the length of an optimal game is an invariant of $G.$ If $k$
cops play on a graph with $k \ge c(G),$ we denote this invariant $\mathrm{capt}_k(G),$ and call it the $k$-\emph{capture time} in $G$. In the case $k=c(G)$, we write $\mathrm{capt}(G)$ and refer to
this as the \emph{capture time} of $G.$\

There is a growing literature on the capture time for graphs with a higher cop number. In~\cite{capt}, the authors proved that if $G$ is cop-win of order $n\ge5,$ then $\mathrm{capt}(G) \leq n-3.$ By
considering small order cop-win graphs, the bound was improved to $\mathrm{capt}(G)\le n-4$ for $n \ge 7$ in~\cite{gav1}. Examples were given of planar cop-win graphs in both \cite{capt,gav1} which
prove that the bound of $n-4$ is optimal. Beyond the case $k=1$, \cite{meh} investigates the capture time of Cartesian grids, which are an important subclass of planar graphs. It was shown in
\cite{meh} that if $G$ is the Cartesian product of two trees, then $\mathrm{capt}(G)=\lfloor \mathrm{diam}(G)/2 \rfloor$.

By results of \cite{over}, bounds were given on the capture time with 3 cops playing on a planar graph.

\begin{theorem}[\cite{over}]\label{thm:3copplanar}
If $G$ is a planar graph of order $n$, then $$\capt_3(G)\leq (\diam(G)+1) n = O(n^2).$$
\end{theorem}

The bound in Theorem~\ref{thm:3copplanar} is an improvement over the bound (for any $G$) that for $k\geq c(G)$, $\capt_k(G)\leq n^{c(G)+1}$~\cite{BI}.  It does not, however, improve the bound of
$n-4$ for (planar or otherwise) cop-win graphs~\cite{capt,gav1}. Further, the $O(n^2)$ bound can be improved to the recent linear bound.

\begin{theorem}[\cite{pt}]\label{ptt}
If $G$ is a planar graph of order $n$, then $\capt_3(G)\leq 2n$.
\end{theorem}

Recent work by Kinnersley \cite{k} and independently by Brandt et al.~\cite{BEUW} shows that there exist families of graphs with cop number $k$ such that $$\capt_k(G) = \Theta(n^{c(G)+1}).$$ We do not know if the
linear bound in Theorem~\ref{ptt} is tight; even providing an example of a planar graph satisfying $\capt(G) > n$ remains open at the time of writing this survey.

If there are many cops playing on a planar graph, then we can reduce the capture time much further. The following result exploits the planar separator theorem of Lipton and Tarjan \cite{LT} (with an
improvement by Alon, Seymour, and Thomas~\cite{AST94}): there is a set of at most $2.13\sqrt{n}$ vertices that separate the graph into two sets of size at most $\frac23 n$. We write
$$r(G)=\min_{x\in V(G)}\max_{y\in V(G)} d_G(x,y)$$ for the \emph{radius} of the graph $G$.

\begin{theorem}[\cite{over}]\label{planarc} If $G$ is a planar graph of order $n$ and $k\geq 12\sqrt{n}$, then $$\capt_k(G)\leq 6r(G)\log n.$$
\end{theorem}

Theorem~\ref{planarc} works even in a version of the game in which the robber is allowed to move \emph{infinitely fast}; that is, they can move to any vertex in the same component of the graph minus
the vertices occupied by the cops.

Similarly, the separator theorem for graphs of genus $g$ by Gilbert, Hutchinson, and Tarjan~\cite{GHT} can be used to obtain the following.

\begin{corollary}[\cite{over}]
If $G$ is a graph of genus $g$ and $k\geq (19+66\sqrt{g})\sqrt{n}$, then $$\capt_k(G)\leq 6r(G)\log n.$$
\end{corollary}

The same holds for the nonorientable genus, for which a separator theorem holds as well (see~\cite{AST94}).

\section{Higher genus graphs}\label{secg}

Less is known about the cop number of graphs with positive genus. Quilliot \cite{Q2} proved the following, also by using the notion of expanding cop territory by induction.

\begin{theorem}[\cite{Q2}]
If $G$ is a graph of genus $g,$ then $c(G)\leq 2g+3.$
\end{theorem}

The best known bound for the cop number of genus $g$ graphs comes from \cite{schr}, and this provides a major conjecture on the cop number.

\bigskip

\noindent \textbf{Schroeder's conjecture}: {\it If $G$ is a graph of genus $g$, then $c(G)\leq  g + 3.$}

\bigskip

Schroeder's conjecture holds for $g=0$ (planar graphs) by Theorem~\ref{theorem:planar}. Essentially, if the conjecture holds, then we would need one additional cop beyond the planar case for each handle of
our surface. In the same paper where his conjecture was stated, Schroeder proved the following bound which was until recently the best known upper bound in general.

\begin{theorem}[\cite{schr}]\label{s2}
If $G$ is a graph of genus $g,$ then $$c(G)\leq \left \lfloor \frac{3g}{2} \right\rfloor +3.$$
\end{theorem}

As an application of Theorem~\ref{s2}, we note that the cop number of toroidal graphs is at most 4. In the same paper, it is shown that the cop number of any graph of genus 2 is at most 5. The conjecture,
therefore, remains open for all graphs with genus at least 3.

Recent work by Bowler, Erde, Lehner, and Pitz \cite{BELP} gives the first improvement since 2001 on the multiplicative constant in the bounds on the cop number for graphs with a given genus.

\begin{theorem}[\cite{BELP}]
If $G$ is a graph of genus $g,$ then for some constant $D$ $$c(G)\leq \left \lfloor \frac{4g}{3} \right \rfloor +D.$$
\end{theorem}

Andreae~\cite{A} first asked whether toroidal graphs do in fact have cop number at most 3, and this problem was referenced in \cite{schr}. To this day, no example of a toroidal graph is known with cop
number 4. Hence, we have the following conjecture.

\bigskip

\noindent \textbf{Andreae-Schroeder conjecture}: {\it Any toroidal graph has cop number at most $3$.}

\bigskip

An interesting open question is to examine bounds on the capture time for higher genus graphs. If Schroeder's conjecture holds, then if $\gamma(G) = g,$ we would have that $\capt_k(G) = O(n^{g+4})$ for every $k\ge g+3$.

We now turn to the non-orientable case. The first result in this direction was by Andreae \cite{A} who proved the following.

\begin{theorem}[\cite{A}]\label{thm:nonorigenus}
If $G$ is a graph of non-orientable genus $g,$ then $$c(G)\leq \binom{\left \lfloor \frac{7}{2} + \sqrt{6g +\frac{1}{4}} \right \rfloor}{2}.$$
\end{theorem}

This result is a corollary of a more general result of Andreae \cite{A}.

\begin{theorem}[\cite{A}]
Let $H$ be a graph with minimum degree at least $2$ and maximum degree $\Delta$. Then every graph that does not contain $H$ as a minor satisfies $$c(G)\le |E(H)| - \Delta.$$
\end{theorem}

The bound of Theorem \ref{thm:nonorigenus} was improved upon in an unpublished work by Nowakowski and Schroeder.

\begin{theorem}[\cite{NS}]
If $G$ is a graph of non-orientable genus $g,$ then $$c(G)\leq 2g+1.$$
\end{theorem}

This bound is tight for the projective plane ($g=1$), where the tightness is witnessed by the Petersen graph.

The best known upper bound in the non-orientable case was found by Clarke, Fiorini, Joret, and Theis \cite{clarke}. Define $c(g)$ to be the maximum of the cop numbers of graphs of genus $g$ and
$\widetilde{c}(g)$ to be the maximum of the cop numbers of graphs of non-orientable genus $g$.

\begin{theorem}[\cite{clarke}]\label{thm:Clarke nonori}
If $g$ is a positive integer, then $$c(\lfloor g/2 \rfloor) \le \widetilde{c}(g) \le c(g-1).$$ In particular, $$\widetilde{c}(g) \leq
\frac{3}{2}(g+1).$$
\end{theorem}

It is also conjectured in \cite{clarke} that the lower bound in Theorem \ref{thm:Clarke nonori} is always tight.

\bigskip

\noindent \textbf{Clarke-Fiorini-Joret-Theis conjecture}: {\it If $g$ is a positive integer, then $$\widetilde{c}(g) = c(\lfloor g/2 \rfloor).$$}

\section{Lower bounds}\label{seclow}

In all of the previous discussion, mainly upper bounds on $c(g)$ and $\widetilde{c}(g)$ were discussed. Not much is known about lower bounds. The notes of one of the authors \cite{MoharNotes} dating
back to 2008 may be the only source.

The simplest and essentially the only known way to prove a lower bound on the cop number is when a graph has girth at least 5. In that case, the minimum degree is a lower bound. By taking a graph of
genus at most $g$ and girth 5 whose minimum degree is as large as possible, we may derive a lower bound on $c(g)$.

Let $q$ be a prime power, and let $G_q$ be the bipartite incidence graph of the finite projective plane of order $q$. Then $n=|G_q|=2(q^2+q+1)$, the graph has girth 6, and it is $\delta$-regular,
where $\delta=q+1$. Note that the genus of any graph is always smaller than the number of edges; thus, the genus of $G_q$ is less than $g = (q^2+q+1)(q+1)$. This shows that
\begin{equation}
   c(g) \ge \Omega(g^{1/3}). \label{eq:cuberootlowerbound}
\end{equation}
The above bound was only argued for certain values of $g$, but since the primes are dense enough (for example, by the Bertrand-Chebyshev theorem there is a prime between any pair of integers $n$ and
$2n$), the lower bound (\ref{eq:cuberootlowerbound}) holds for every $g$.

By using Euler's formula, it is straightforward to see that a graph $G$ of girth 5 and minimum degree $\delta\ge7$ has genus at least
$$
    \tfrac{1}{12}(\delta-6)|G|.
$$
Since the number of vertices of such a graph is at least $\delta^2$, lower bounds on $c(g)$ using graphs of girth 5 and their minimum degree cannot go beyond the bound of $\Theta(g^{1/3})$.

However, there is a better lower bound, which uses random graphs (see \cite{BollobasKunLeader} or \cite{LuczakPralat}). We say that an event in a probability space holds \emph{asymptotically almost
surely} (\emph{a.a.s.}) if its probability tends to one as $n$ goes to infinity.

\begin{theorem}[Bollob\'as, Kun, and Leader \cite{BollobasKunLeader}]
\label{thm:random} If $p=p(n)\ge 2.1\log(n)/n$, then a.a.s.\ we have that
$$
   (np)^{-2} n^{1/2-o(1)} \le c(G_{n,p}) \le 160000 \sqrt{n} \log(n).
$$
\end{theorem}

\begin{theorem}[\cite{MoharNotes}] \label{mn}
For every $\varepsilon>0$ and every sufficiently large $g$, there is a graph $G$ of (nonorientable) genus $g$ whose cop number is bounded as
$$g^{\frac{1}{2}-\varepsilon} \le c(G) \le g^{\frac{1}{2}+\varepsilon}.$$
In other words,
$$c(g)\geq \widetilde c(G) \ge g^{\frac{1}{2}-o(1)}.$$
\end{theorem}

\begin{proof}
Let $p=p(n)=\tfrac{5}{2}\log(n)/n$. The random graph a.a.s.\ $G=G(n,p)$ has less than $2n\log n$ edges; thus, its genus $g$ is smaller than $2n\log n$. By Theorem~\ref{thm:random}, its cop number is
a.a.s.\ at least
$$c(G) \ge (np)^{-2} n^{1/2-o(1)} \ge g^{\frac{1}{2}-\varepsilon}.$$

The upper bound holds for the same graph. The proof uses the same two results (just the opposite bounds), where the bigger constants and the $\log(n)$ factors can be hidden in $g^{\varepsilon}$ when $g$ is large enough.
\end{proof}

It is also suggested that random graphs indeed provide the worst behaviour and the following conjecture is proposed in \cite{MoharNotes}.

\bigskip

\noindent
\textbf{Mohar's conjecture}: {\it $c(g)= g^{\frac{1}{2}+o(1)}$\ and \ $\widetilde c(g)= g^{\frac{1}{2}+o(1)}$.
In other words, for every $\varepsilon >0$ there exists $g_0$ such that for every $g\ge g_0$,
$$g^{\frac{1}{2}-\varepsilon} < c(g) < g^{\frac{1}{2}+\varepsilon}\quad \textrm{and} \quad
g^{\frac{1}{2}-\varepsilon} < \widetilde c(g) < g^{\frac{1}{2}+\varepsilon}.$$}

\section{Miscellaneous topics}\label{secm}

We next consider a number of topics related to Cops and Robbers games and their variants played on surfaces. We focus on the lazy cop number, Zombies and Survivors, and the cop number of directed
planar graphs.

\subsection{Lazy Cops and Robbers}

There are many variants of the game of Cops and Robbers, where the cops and robbers may have some advantage or disadvantage in gameplay (see \cite{bonato} for some of these). In the game of
\emph{Lazy Cops and Robbers}, the rules are analogous to the classic game with the exception that only one cop may move at a time. Hence, Lazy Cops and Robbers is a game more akin to chess or
checkers. The corresponding parameter to the cop number is the \emph{lazy cop number}, written $c_L(G).$ Clearly, $c_L(G) \ge c(G)$ and the bound is not tight in general. This game and parameter was
first considered by Offner and Ojakian~\cite{oo}, who included bounds on the lazy cop number of hypercubes. These bounds were sharpened in \cite{lazy1} using the probabilistic method.

The following bound given in \cite{lazy2} gives asymptotic upper bound on $c_L$ for graphs of genus $g$.  It exploits the Gilbert, Hutchinson, and Tarjan separator theorem \cite{GHT}.

\begin{theorem}\label{thm:bded_genus}
For every $n$-vertex graph $G$ of genus $g$ we have $$c_L(G) \le 60\sqrt{gn} + 20\sqrt{2n}= O(\sqrt{gn}).$$
\end{theorem}

It is not known whether the bound in Theorem~\ref{thm:bded_genus} is asymptotically tight, even in the case of planar graphs.  Further, and most interestingly, we are not presently aware of any
families of planar graphs on which the lazy cop number grows as an unbounded function. Recent work of Gao and Yang \cite{boting} gives a non-trivial example of a planar graph $G$ such that $c_L(G) \ge
4$ (in contrast to the upper bound of 3 given Theorem~\ref{theorem:planar}). Note that this answers a question posed in \cite{su}, where it was shown that the graph $K_3 \,\square\, K_3$ (that is,
the Cartesian product of $K_3$ with itself) is the
graph of smallest possible order having lazy cop number 3.

\subsection{Zombies and Survivors}

In the game of \emph{Zombies and Survivors}, suppose that $k$ \emph{zombies} (analogous to the cops) start the game on random vertices of $G$; each zombie, independently, selects a vertex uniformly
at random to begin the game. Then the \emph{survivor} (analogous to the robber) occupies some vertex of $G$. As zombies have limited intelligence, in each round, a given zombie moves toward the
survivor along a shortest path connecting them. If there is more than one neighbor of a given zombie that is closer to the survivor than the zombie's current position, then they move to one of these vertices chosen
uniformly at random. Each zombie moves independently of all other zombies. As in Cops and Robbers, the survivor may move to another neighboring vertex, or pass. The zombies win if one or more of them
\emph{eat} the survivor; that is, land on the vertex currently occupied by the survivor. The survivor, as survivors should do in the event of a zombie attack, attempts to survive by applying a
strategy that minimizes the probability of being eaten. Note that there is no strategy for the zombies; they merely move on geodesics toward the survivor in each round. In this sense, Zombies and
Survivors is a one-person game.

The probabilistic version of Zombies and Survivors was first introduced in \cite{zombies-bmpp}. The random zombie model was inspired by a deterministic version of this game (with similar rules, but
the zombies may choose their initial positions, and also choose which shortest path to the survivor they will move on) first considered in~\cite{zombies-hm}.

Let $s_k(G)$ be the probability that the survivor wins the game, provided that they follow the optimal strategy. Note that $s_k(G)$ is a nondecreasing function of $k$; that is, for every $k \ge 1$,
we have that $s_{k+1}(G) \le s_k(G)$, and $s_k(G) \to 0$ as $k\to \infty$. The latter limit follows since the probability that each vertex is initially occupied by at least one zombie tends to 1 as
$k \to \infty$. Define the \emph{zombie number} of a graph $G$ by
$$
z(G) = \min \{ k : s_k(G) \le 1/2 \} .
$$
Note that the zombie number is well-defined as $s_0(G)=1$, $s_k(G)$ is decreasing in $k$ and $$\lim_{k\to\infty}s_k(G)=0.$$ In particular, $z(G)$ is the smallest number of zombies such that the
probability that they eat the survivor is at least 1/2. Note that $z(G)\ge c(G)$.

In \cite{zombies-bmpp}, the zombie number was computed asymptotically for cycles, hypercubes, and incidence graphs of projective planes. For the Cartesian grid, two zombies are sufficient to eat the
survivor (and the cop number of this graph is 2). However, for grids spanning the surface of a torus, a much more complex situation arises.  Let $T_n$ be the \emph{toroidal $n\times n$ grid}, which
is isomorphic to $C_n \,\square\, C_n$ (that is, a 4-regular quadrangulation of the torus). The following lower bound for the zombie number of $T_n$ was proved in \cite{zombies-bmpp}.

\begin{theorem}[\cite{zombies-bmpp}]\label{chz:thm:torus}
If $\omega = \omega(n)$ is a function tending to infinity as $n\to \infty$, then a.a.s.\ $z(T_n) \ge \sqrt n/(\omega(n)\log n)$.
\end{theorem}

We note that $z(T_n) = O(n^2 \log n)$. To see this, suppose that the game is played against $k = 3 n^2 \log n$ zombies. It is straightforward to see that a.a.s.\ every vertex is initially occupied by
at least one zombie and if so, the survivor is eaten immediately. Indeed, the probability that at least one vertex is not occupied by a zombie is at most $n^2 (1-1/n^2)^k \le n^2 \exp(-k/n^2) = 1/n =
o(1)$. However, no quadratic bounds in $n$ are known for the zombie number of toroidal grids.

\subsection{Directed graphs}

Cops and Robbers is played in directed graphs in an analogous way as in the undirected case, with the exception that we must move in the orientation of directed edges. Note that a robber wins if they
occupy a source. Hence, we consider the cop number of strongly connected directed graphs.

Challenges emerge quickly when considering the cop number of directed graphs. For example, while cop-win graphs are precisely the dismantlable ones \cite{nw,q}, there is no known characterization of
cop-win directed graphs. Further, the directed analogue of the Isometric Path Lemma is not applicable, as the cop may not be able to simply move back and forth on the path.

In \cite{fkl}, it was shown using the probabilistic method that if $G$ is a strongly connected directed graph, then
$$c(G) = O\left( \frac{n(\log \log n)^2}{\log n} \right).$$
For planar graphs, this bound was sharpened in \cite{LO} where it was proved that $c(G) = O(\sqrt{n})$ if $G$ is strongly connected and planar. The same result holds for digraphs of bounded genus.
For directed planar graphs, the cop number is not well understood. In \cite{LO}, an example of a strongly connected planar directed graph $G$ was given with $c(G) \ge 4$ (contrasting with
Theorem~\ref{theorem:planar}). We do not know if the cop number of strongly connected planar directed graphs is bounded by a constant. Note that the game on Eulerian digraphs generalizes the game
played on undirected graphs, since the game on any undirected graph $G$ is equivalent to the game on the Eulerian digraph obtained from $G$ by replacing each edge with a pair of oppositely oriented
arcs joining the same pair of vertices.

Recent work by Hosseini and Mohar~\cite{HM} considers the cop number of several ``nice'' orientations of toroidal grids. More precisely, they consider 4-regular quadrangulations of the torus and the Klein bottle,
subject to certain constraints about orientations of the edges. One such orientation of the toroidal grid is where each row and column are oriented either all left or right or all up and down. In all
cases they investigate, the cop number of such directed graphs is either 3 or 4. In the forthcoming paper \cite{maza}, it is shown that the cop number of arbitrary ``straight-ahead orientations'' of
4-regular quadrangulations is bounded above by 404.

\section{Summary of conjectures and open problems}

For the utility of readers, we gather together all the major unresolved questions presented in the survey. We include citations where relevant.

\begin{enumerate}

\item Classify the planar graphs with cop number $i$, for $1\le i \le 3$. We also pose the analogous problem for outerplanar graphs with $i=1,2.$

\item \textbf{Conjecture} \cite{bevbon}: The smallest order of a $3$-cop-win planar graph is $20$ and the corresponding graph is the dodecahedron.

\item Determine a tight bound on the capture time of planar graphs with cop number 2 and 3.

\item \textbf{Schroeder's conjecture} \cite{schr}: If $G$ is a graph of genus $g$, then $c(G)\leq  g + 3.$

\item \textbf{Andreae-Schroeder conjecture} \cite{A,schr}: Every toroidal graph has cop number at most 3.

\item \textbf{Mohar's conjecture}: For every $\varepsilon >0$ there exists $g_0$ such that for every $g\ge g_0$,
$$g^{\frac{1}{2}-\varepsilon} < c(g) < g^{\frac{1}{2}+\varepsilon}\quad \textrm{and} \quad
g^{\frac{1}{2}-\varepsilon} < \widetilde c(g) < g^{\frac{1}{2}+\varepsilon}.$$

\item \textbf{Clarke-Fiorini-Joret-Theis conjecture} \cite{clarke}: If $g$ is a positive integer, then $$\widetilde{c}(g) = c(\lfloor g/2 \rfloor).$$

\item Is the lazy cop number bounded by a constant on the class of all planar graphs? See \cite{lazy2}.

\item Determine the zombie number of the toroidal grid. Note that it is open to find either the optimal upper or lower bound. An easier problem may be to derive a quadratic upper bound on the
    zombie number. See \cite{zombies-bmpp}.

\item Determine if there are Eulerian planar directed graphs with arbitrarily large cop number. See \cite{HM,LO}.

\end{enumerate}

\end{document}